\documentclass[preprint,12pt]{elsarticle}

\usepackage{comment}
\usepackage{url}
\usepackage{graphics}
\usepackage{graphicx}
\usepackage{xcolor}
\usepackage{epsfig}
\usepackage{array}
\usepackage{multicol}
\usepackage{amsthm,amsmath,amsfonts,amssymb}

\usepackage{lineno}
\usepackage{hyperref}
\usepackage{amsfonts}
\usepackage{amsmath}
\usepackage{amssymb}% http://ctan.org/pkg/amssymb
\usepackage{pifont}% http://ctan.org/pkg/pifont

\def\Q{\ns Q}

\def\v{\mbox{\boldmath $v$}}

\def\vec0{\mbox{\boldmath $0$}}

\def\M{\mbox{\boldmath $M$}}

\def\M{\mbox{\boldmath $M$}}

\def\Q{\mbox{\boldmath $Q$}}

\journal{Discrete Mathematics}

\begin{document}

\begin{frontmatter}

%% Title, authors and addresses

%% use the tnoteref command within \title for footnotes;
%% use the tnotetext command for the associated footnote;
%% use the fnref command within \author or \address for footnotes;
%% use the fntext command for the associated footnote;
%% use the corref command within \author for corresponding author footnotes;
%% use the cortext command for the associated footnote;
%% use the ead command for the email address,
%% and the form \ead[url] for the home page:
%%
\title{Bounds in radial Moore graphs of diameter 3}
%% \tnotetext[label1]{}

\author[mat]{Jesús M. Ceresuela}
\ead{jesusmiguel.ceresuela@udl.cat}
\author[mat]{Nacho López}
\ead{nacho.lopez@udl.cat}

%% \ead{email address}
%% \ead[url]{home page}

%% \cortext[cor1]{}

%% \fntext[label3]{}

%% \title{Unncessarily Complicated Research Title}

%% use optional labels to link authors explicitly to addresses:
%% \author[label1,label2]{<author name>}

\address[mat]{Deptartament de Matemàtica, Universitat de Lleida, Lleida, Spain.}
%% \address[label2]{<address>}

\theoremstyle{plain}   % Cal carregar el paquet theorem.sty o amsthm.sty
\newtheorem{theorem}{Theorem}[section]
\newtheorem{proposition}[theorem]{Proposition}
\newtheorem{corollary}[theorem]{Corollary}
\newtheorem{lemma}[theorem]{Lemma}
\newtheorem{definition}[theorem]{Definition}
\newtheorem{conjecture}[theorem]{Conjecture}
\newtheorem{question}{Question}
\newtheorem{example}[theorem]{Example}
\newtheorem{problem}[theorem]{Problem}
\newtheorem{observation}[theorem]{Observation}
\newcommand{\rad}{\mathrm{rad}}
\begin{abstract}

Radial Moore graphs are approximations of Moore graphs that preserve the distance-preserving spanning tree for its central vertices. One way to classify their resemblance with a Moore graph is the status measure. The status of a graph is defined as the sum of the distances of all pairs of ordered vertices and equals twice the Wiener index. In this paper we study upper bounds for both the maximum number of central vertices and the status of radial Moore graphs. Finally, we present a family of radial Moore graphs of diameter $3$ that is conjectured to have maximum status.
\end{abstract}

\end{frontmatter}

\section{Introduction}

The number of vertices of a graph that has maximum degree $d$ and maximum diameter $k$ is upper bounded by
\begin{equation}
    M(d,k)=1+d+d(d-1)+d(d-1)^2+\cdots+d(d-1)^{k-1}.
\end{equation}
Hoffman and Singleton gave the name of Moore graphs to graphs holding the mentioned restrictions and attaining the bound \cite{Hoffman1960}. Also the bound is known nowadays as the Moore bound.

The existence of Moore graphs is a well-studied and almost-closed problem (see \cite{MS12}). They must be regular of degree $d$ and the fact that Moore graphs exist for few values of $d$ and $k$ entails the question of how close can you get to this structures with feasible graphs. One way to do it is to preserve the Moore bound as the order of the candidate graphs as well as the regularity of the graph, and then relax the diameter condition a bit, by allowing the existence of some vertices in the graph with eccentricity just $k+1$. These graphs are named radial Moore graphs. More specifically, given two integers $d$ and $k$, a radial Moore graph is a $d$-regular graph of order $M(d,k)$ with radius $k$ and diameter $k+1$.

The existence of such graphs has been thoroughly studied. In \cite{CCEGL2010} the authors give a construction of a radial Moore graph of diameter $3$, for all values of $d$. Radial Moore graphs of diameter $4$ also exist for every value of $d$ (see \cite{EXOO20121507,Knor2007SmallRadially,knor2012}). The existence of radial Moore graphs for any value of $k$ and large enough $d$ has been proved in \cite{GOMEZ201515}. This problem has also been addressed in the directed \cite{knor1996} and mixed \cite{CLC2023} version.

Radial Moore graphs are conceived to approach Moore graphs. There are many radial Moore graphs in some cases ($1062$ for $(d,k)=(3,3)$, see \cite{CCEGL2010}). Thus, a natural question would be, given the parameters $d$ and $k$, which is (\textit{are}) the closest radial Moore graph (\textit{graphs}) to being a Moore graph. To quantify the difference between two graphs, a closeness measure must be defined. A very suitable option is the status norm used in \cite{CCEGL2010} and in \cite{CLC2023} that sorts the radial Moore graphs according to the status vector. This measure is very convenient since every vertex in a Moore graph has the same status $s_{d,k}$, where 
    \begin{equation}\label{eq:stat}
    s_{d,k}=\frac{d\left( k(d-2)(d-1)^k-(d-1)^k +1 \right)}{(d-2)^2},
    \end{equation}
(see \cite{CLC2023}). The value of $s_{d,k}$ can be obtained even in the cases when Moore graphs do not exist, this allows to classify radial Moore graphs for all values of $d$ and $k$, regardless of the existence of a Moore graph. This is especially important since Moore graphs do not exist for most parameters. Besides, the status of a vertex of a radial Moore graph is lower bounded by $s_{d,k}$. Thus, according to this measure, finding the closest radial Moore graph is equivalent to find the radial Moore graph with minimal status. 

\subsection*{Organization of the paper}

Section \ref{sec:CentralV} is devoted to obtain a bound on the maximum number of central vertices that a radial Moore graph may have. In section \ref{sec:VertexBound} we highlight some results on the status for vertices of radial Moore graphs, including a general lower bound for the total status, we also give the construction and some properties of a family $G_d$ of radial Moore graphs of diameter $3$ that is conjectured to maximize the status. Finally, some concluding remarks and open problems are presented in section \ref{sec:conc}.

\subsubsection*{Terminology and notation}

Let $G=(V,E)$ be a connected graph with vertex set $V$ and edge set $E$. A {\em walk\/} of length $\ell\geq 0$ from $u$ to $v$ is a sequence of $\ell+1$ vertices, $u_0u_1\dots u_{\ell-1}u_\ell$, such that $u=u_0$, $v=u_\ell$ and each pair $u_{i-1}u_i$, for $i=1,\ldots,\ell$, is an edge of $G$. A walk whose vertices are all different is called a {\em path}. The length of a shortest path from $u$ to $v$ is the {\it distance\/} from $u$ to $v$, and it is denoted by $d(u,v)$. The set of vertices in $G$ at distance $i$ from vertex $v$ will be denoted by $\Gamma_i(v)$. The sum of all distances from a vertex $v$, $s(v)=\sum_{u\in V} d(v,u)$, is referred to as the {\em status\/} of $v$ (see \cite{BuckHara}). We define the {\em status vector\/} of $G$, $\mathbf{s}(G)$, as the vector constituted by the status of all its vertices. Usually, when the vector is long enough, we denote it with a short description using superscripts, that is, $\mathbf{s}(G):s_1^{n_1},s_2^{n_2},\dots,s_k^{n_k}$, where $s_1>s_2> \dots >s_k$, and $n_i$ denotes the number of vertices having $s_i$ as its local status, for all $1 \leq i \leq k$. Similarly, the {\em status \/} of a graph $G$, $s(G)$ is the sum of the components of its status vector. The {\em eccentricity\/} of a vertex $u$ is the maximum distance from $u$ to any vertex in $G$. A {\em central vertex} is a vertex having minimum eccentricity. The minimum eccentricity of all vertices is the {\em radius} of G. The maximum distance between any pair of vertices is the {\it diameter} of $G$.

\section{Bound on the maximum number of central vertices in a radial Moore graph}\label{sec:CentralV}

There are two kinds of vertices in a radial Moore graph of radius $k$: {\em 
 central vertices} are those vertices with eccentricity $k$ and they behave like every vertex in a Moore graph, in the sense that their corresponding distance-preserving spanning tree is the same. Besides, {\em non-central vertices} are those vertices with eccentricity $k+1$, and their corresponding distance-preserving spanning tree has one more layer containing some vertices at distance $k+1$. Most of the known radial Moore graphs have just one central vertex, but the optimal ones (those graphs closest to the Moore graphs in some kind of measure) have many central vertices. Since every vertex is a central vertex in a Moore graph, we would like to know how many central vertices may exist at most in a radial Moore graph. We will denote this value as ${\cal C}(d,k)$, following the notation in \cite{LG10}. The exact value of ${\cal C}(d,k)$ is known for few cases, where the total population of radial Moore graphs has been computed, that is, ${\cal C}(3,2)=4, {\cal C}(3,3)={\cal C}(4,2)=8$. The lower bound ${\cal C}(d,2)\geq 2d$ for $d>3$, is provided in \cite{CCEGL2010} by construction of a family of graphs of degree $d$ having $2d$ central vertices. Besides, the upper bound 
\[
{\cal C}(3,k) \leq 3\cdot 2^k+\frac{1}{5}\cdot 2^{1-k}\big((1-\sqrt{5})^k(\sqrt{5}-5)-(1+\sqrt{5})^k(\sqrt{5}+5)\big)
\]
is obtained in \cite{LG10} by exploiting the relationship between central and non-central vertices in a cubic radial Moore graph. Here we extend these ideas to provide a general upper bound for any degree $d$. To this end, we show the following results presented in \cite{CLC2023} and \cite{GL08} related to the mixed setting, but they can be applied in the particular case of undirected graphs. 

\begin{proposition}[\cite{CLC2023}]\label{ncneigbors}
    Let $G$ be a radial Moore graph and $w\in V(G)$ be a central vertex of $G$. Suppose $w$ is adjacent to a non-central vertex $u\in V(G)$. Then there exists another vertex in the neighborhood of $w$ different from $u$ which is also non-central.
\end{proposition}

\begin{proposition}[\cite{GL08}] \label{nachgim}
    Every non-central vertex in a radial Moore graph must have at least two non-central neighbours.
\end{proposition}

These results will be combined to obtain a new upper bound for ${\cal C}(d,k)$. %, that leads directly to a lower bound in the status of a radial Moore graph

\begin{proposition}\label{prop:centralvertices}
    Given $d>3$ and $k\geq2$, the number of central vertices ${\cal C}(d,k)$ in a radial Moore graph of degree $d$ and radius $k$ satisfies ${\cal C}(d,k) \leq \sum_{j=1}^k a(j)+a'(j)$, where $a(j)$ and $a'(j)$ hold the recurrence equation
        \begin{equation*}
        \left(\begin{array}{l}
            a(j)\\
            a'(j) \\
            b(j) \\
            b'(j) \\
        \end{array}\right)=
\left(\begin{array}{cccc}
d - 1 & d - 2 & 0 & 0 \\
0 & 0 & d - 3 & d - 2 \\
0 & 1 & 0 & 0 \\
0 & 0 & 2 & 1
\end{array}\right) 
        \left(\begin{array}{l}
            a(j-1) \\
            a'(j-1) \\
            b(j-1) \\
            b'(j-1) \\
        \end{array}\right)
   \end{equation*}
   
Moreover, the number of non-central vertices is at least of the order $\mathcal{O}(\sqrt{d^k})$ for $d$ large enough.
\end{proposition}

\begin{proof}
Let $u$ be a non-central vertex of a radial Moore graph. Proposition \ref{nachgim} forces two of the neighbors to be non-central, but since $d>3$, at least one of the neighbors is not forced to be non-central and we can consider it to be central. Let's call this vertex $w$. Since it is a central vertex, the spanning distance-preserving tree hanging at $w$ will have the structure of a Moore tree. Let's consider this tree. Due to proposition \ref{ncneigbors}, $v$ cannot be the only non-central child of $w$, and thus there must exist another non-central vertex at distance one from $w$. The rest of the $d-2$ children vertices of $w$ can still be considered central. Now we can produce the rest of the tree,  maximizing the number of central vertices, according to the following rules.
    \begin{itemize}
        \item Central vertices hanging from a central vertex will be named $a$-vertices, and their $d-1$ children have no restrictions, so we can consider them to be central vertices (and thus, also $a$-vertices).
        \item Central vertices hanging from non-central vertices will be named $a'$-vertices and are forced to have one non-central child according to \ref{ncneigbors} and the other $d-2$ children vertices can be assumed to be central.
        \item Non-central vertices hanging from a central vertex will be named $b$-vertices and must have at least two non central children vertices due to proposition \ref{nachgim}. The remaining $d-3$ vertices can be considered central vertices
        \item Non-central vertices hanging from non-central vertices will be named $b'$-vertices and must have at least one non-central child according to \ref{nachgim} and the other $d-2$ children vertices can be assumed to be central.
    \end{itemize}
    \begin{figure}[htbp]
    \centering\includegraphics[width=0.77\textwidth]{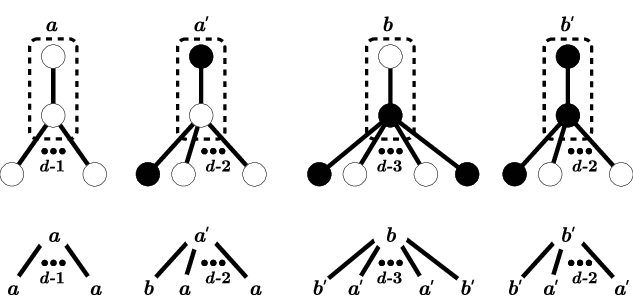}
    \caption{The four different kinds of vertices in the distance preserving tree hanging from $w$.}
    \label{recsk}
    \end{figure} 
    
    For a sketch of these 4 rules, see Fig. \ref{recsk}. Now let $a(j)$, $a'(j)$, $b(j)$ and $b'(j)$ be the number of $a$, $a'$, $b$, and $b'$ vertices respectively at distance $j$ from $w$. The four rules can be encoded in the recursive system given at beginning of the proposition, and hence the number of central vertices ${\cal C}(d,k)$ is at most $\sum_{j=1}^k a(j)+a'(j)$. The matrix of this system is 

    \begin{equation*}
\M=\left(\begin{array}{cccc}
d - 1 & d - 2 & 0 & 0 \\
0 & 0 & d - 3 & d - 2 \\
0 & 1 & 0 & 0 \\
0 & 0 & 2 & 1
\end{array}\right). 
   \end{equation*}
It has characteristic polynomial $
\left(x-(d-1)\right)\left(x^3 - x^2 - (d-3)x - (d-1)\right)$,
that is, $\M$ has eigenvalues $d-1$ and the roots of $q(x)=x^3 - x^2 - (d-3)x - (d-1)$. This cubic polynomial has discriminant $\Delta=-\frac{1}{27}(d^3-20d^2+56d-44)$ and hence $q(x)$ has just one real root whenever $d\leq 16$ and three real roots otherwise. Let $\alpha:=\alpha(d),\beta:=\beta(d),\gamma:=\gamma(d)$ be the three roots of $q(x)$. Now we will see that all these three eigenvalues are negligible if we compare them with the `leading' eigenvalue $d-1$, for $d$ large enough. Indeed, for $d>16$ we have that all three real roots $\alpha(d),\beta(d),\gamma(d)$ must be located into the interval of real numbers $\left(\frac{1}{3}(1-2\sqrt{3d-8},\frac{1}{3}(1+2\sqrt{3d-8}\right)$, due to Laguerre's bound (see \cite{zbMATH01820648}). In fact, computing their explicit values using Cardano's formulas ($
\alpha=
\frac{1}{3} \left(1+\sqrt[3]{18 \, d -26 + 27 \, \sqrt{\Delta}} +\sqrt[3]{18 \, d -26 - 27 \, \sqrt{\Delta}}\right)$, an the other two roots can be computed through the relationships $\beta+\gamma=1-\alpha$ and $\beta\cdot \gamma=(d-1)/\alpha$) shows that 
$$
\lim \frac{\alpha(d)}{\sqrt{d}}=\lim \frac{\beta(d)}{-\sqrt{d}}=1 \ \textrm{and} \ \lim \gamma(d)=-1.
$$
So, they have order at most $\mathcal{O}(\sqrt{d})$. The solution of the recurrence system can be described in terms of these roots:
    \begin{equation*}
        \left(\begin{array}{c}
            a(j)\\
            a'(j) \\
            b(j) \\
            b'(j) \\
        \end{array}\right)=\Q
\left(\begin{array}{cccc}
(d - 1)^{j-1} & 0 & 0 & 0 \\
0 & \alpha^{j-1} & 0 &  0 \\
0 & 0 & \beta^{j-1} & 0 \\
0 & 0 & 0 & \gamma^{j-1}
\end{array}\right)\Q^{-1} 
        \left(\begin{array}{c}
            a(1) \\
            a'(1) \\
            b(1) \\
            b'(1) \\
        \end{array}\right)
   \end{equation*}

where $\Q$ is the matrix whose columns are the eigenvectors corresponding to the roots of $q(x)$, and $a(1)=d-2,a'(1)=0,b(1)=2$ and $b'(1)=0$. Notice that the number of the $a$-vertices does not change the number of the other kind of vertices since they only produce $a$-vertices. This is reflected in the eigenvector of $d-1$ which is $\v=(1,0,0,0)^t$ since $\M \v=(d-1)\v$. This means that the number of non central vertices given by $\sum_{j=1}^k b(j)+b'(j)$ only depends on the roots $\alpha,\beta,\gamma$ whose sum from $1$ to $k$ has order at most $\mathcal{O}(\sqrt{d^k})$ for $d$ large enough. 
\end{proof}

Note that for $d\leq 16$ we obtain two complex roots that are bounded by $1+\max\{1,\frac{d-3}{d-1}\} \leq 2$ due to Cauchy's bound (see \cite{zbMATH01820648}). Again the leading root $d-1$ is taking into account only in the counting of $a$-vertices and hence the number of non-central vertices have order negligible when $k$ is large enough. We do not depict here the exact solutions of the recurrence, because they are tedious to describe, although curiously enough, the case $d=7$ is the only one where their representation is fair enough to write it down: In this case, the roots of $q(x)$ become $\alpha=3,\beta=-i-1,\gamma=i-1$, and hence
\begin{equation*}
\M=\left(\begin{array}{cccc}
6 & 5 & 0 & 0 \\
0 & 0 & 4 & 5 \\
0 & 1 & 0 & 0 \\
0 & 0 & 2 & 1
\end{array}\right)=\Q\left(\begin{array}{cccc}
6 & 0 & 0 & 0 \\
0 & 3 & 0 & 0 \\
0 & 0 & -i - 1 & 0 \\
0 & 0 & 0 & i - 1
\end{array}\right)\Q^{-1}
   \end{equation*}
where $\Q$ is the matrix of the corresponding eigenvectors of $\M$, which is:
\[
\Q=\frac{1}{5}\left(\begin{array}{cccc}
5 & 5 & 5 & 5 \\
0 & -3 & -i - 7 & i - 7 \\
0 & -1 & -3 i + 4 & 3 i + 4 \\
0 & -1 & 4 i - 2 & -4 i - 2
\end{array}\right),
\]
from where we obtain,
\[
\begin{array}{lll}
a(j)& = &7 \cdot 6^{j - 1} - \frac{60}{17} \cdot 3^{j - 1} - \left(\frac{1}{17} i - \frac{13}{17}\right) \, \left(i - 1\right)^{j - 1} + \left(\frac{6}{17} i - \frac{7}{17}\right) \, \left(-i - 1\right)^{j}, \\
a'(j)& =& \frac{36}{17} \cdot 3^{j - 1} + \left(\frac{4}{17} i - \frac{18}{17}\right) \, \left(i - 1\right)^{j - 1} - \left(\frac{7}{17} i - \frac{11}{17}\right) \, \left(-i - 1\right)^{j}, \\
b(j)&=&\frac{12}{17} \cdot 3^{j - 1} + \left(\frac{7}{17} i + \frac{11}{17}\right) \, \left(i - 1\right)^{j - 1} + \left(\frac{9}{17} i - \frac{2}{17}\right) \, \left(-i - 1\right)^{j}, \\
b'(j)&=&\frac{12}{17} \cdot 3^{j - 1} - \left(\frac{10}{17} i + \frac{6}{17}\right) \, \left(i - 1\right)^{j - 1} - \left(\frac{8}{17} i + \frac{2}{17}\right) \, \left(-i - 1\right)^{j}.
\end{array}
\]
The value of $\sum_{j=1}^k a(j)+a'(j)$ yields the bound,

\[
{\cal C}(7,k) \leq \frac{7}{5} \cdot 6^{k} - \frac{12}{17} \cdot 3^{k} - \left(\frac{1}{85} i - \frac{13}{85}\right) \, \left(i - 1\right)^{k} + \left(\frac{1}{85} i + \frac{13}{85}\right) \, \left(-i - 1\right)^{k} - 1.
\]
Notice that its leading coefficient $\frac{7}{5}6^k$ is the same than the corresponding one in the Moore bound $M(7,k)$. This means the number of non-central vertices is of the order $\mathcal{O}(3^k)$. We see it also from the computation of $\sum_{j=1}^k b(j)+b'(j)$, from where we obtain that the number of non-central vertices is at least 
\[
\frac{12}{17} \cdot 3^{k} + \left(\frac{1}{85} i - \frac{13}{85}\right) \, \left(i - 1\right)^{k} - \left(\frac{1}{85} i + \frac{13}{85}\right) \, \left(-i - 1\right)^{k} - \frac{2}{5}
\]
which is precisely the difference between $M(7,k)$ and ${\cal C}(7,k)$. \\

The bound given in Prop. \ref{prop:centralvertices} is constant for $k=2$, where ${\cal C}(d,2) \leq M(d,2)-6$ is obtained for any $d>3$. For other values of $d$ and $k$, ${\cal C}(d,k)$ has been computed and it is shown in table \ref{tab:gammark}.

\begin{table}[h]
\centering
\[
\begin{array}{|c|c|c|c|c|c|}
\hline
k \backslash d & 4 & 5 & 6 & 7 \\
\hline
3 & 41\,(53) & 92\,(106) & 171\,(187) & 284\,(302) \\
4 & 133\,(161) & 388\,(426) & 889\,(937) & 1756\,(1814) \\
5 & 423\,(485) & 1612\,(1706) & 4557\,(4687) & 10716\,(10886) \\
6 & 1327\,(1457) & 6596\,(6826) & 23079\,(23437) & 64804\,(65318) \\
7 & 4093\,(4373) & 26732\,(27306) & 116195\,(117187) & 390364\,(391910) \\
\hline
\end{array}
\]
\caption{Values of the upper bound of ${\cal C}(d,k)$ given in Prop. \ref{prop:centralvertices}. The value of the corresponding Moore bound is in parenthesis.}\label{tab:gammark}
\end{table}

\section{Bounds on the status of vertices in radial Moore graphs}\label{sec:VertexBound}

Let us denote by $\mathcal{RM}(d,k)$ the set of all non-isomorphic radial Moore graphs of degree $d$ and radius $k$. We recall that the status $s(v)$ of a vertex $v$ is the sum of all distances from this vertex $v$. This value is the same for all vertices in a Moore graph, and it is computed 
Of course, the same number is attained for central vertices in a radial Moore graph, but for non-central vertices (having vertices at distance $k+1$) we only know that $s(v)>s_{d,k}$. Hence, a natural question arises which is how far the value of $s(v)$ can go. Here, we establish upper bounds for the value of $s(v)$ in a radial Moore graph of diameter $3$.

\begin{lemma}
    Let $v$ be a vertex in a radial Moore graph of degree $d\geq 3$ and diameter $3$ containing one central vertex. Then, there is at least $d$ vertices at distance $2$ from $v$, that is, $|\Gamma_2(v)|\geq d$. Moreover, if the central vertex is in the neighborhood of $v$ then 
\begin{equation}
    |\Gamma_2(v)|\geq d+ \left\lceil \frac{1+\sqrt{4d-3}}{2}\right\rceil -1,
\end{equation}
if $d\neq 4$, and $|\Gamma_2(v)|\geq 5$ if $d=4$.

\end{lemma}
\begin{proof}
    On one side, since the order of the graph is $d^2+1$ and the graph is $d$-regular the following equality must hold $$|\Gamma_2(v)|+|\Gamma_3(v)|=d^2+1-(|\Gamma_0(v)|+|\Gamma_1(v)|)=d(d-1).$$ On the other side, every vertex in $\Gamma_2(v)$ can be adjacent to a maximum of $d-1$ vertices in $\Gamma_3(v)$, since at least one of its edges must be adjacent to a vertex in $\Gamma_1(v)$. Thus, $$|\Gamma_3(v)|\leq|\Gamma_2(v)|(d-1).$$ Joining both expressions $$d(d-1)=|\Gamma_2(v)|+|\Gamma_3(v)|\leq|\Gamma_2(v)|+|\Gamma_2(v)|(d-1)=d|\Gamma_2(v)|$$ and thus $|\Gamma_2(v)|\geq d-1$.

    So far we have proved that the minimum number of vertices at distance $2$ of $v$ is $d-1$. Now we will prove that it is indeed $d$. Let's suppose that $v$ is a vertex such that $|\Gamma_2(v)|=d-1$. This is equivalent to say that every vertex in $\Gamma_2(v)$ hangs from a single vertex in  $\Gamma_1(v)$ and its remaining edges are incident to different vertices in $\Gamma_3(v)$.
    
    %\jesus{para la demo creo que no hace falta separar por par o impar porque la segunda parte no tiene en cuenta la paridad} Let's consider the subgraph $H$ of $G$ induced by the vertices in $\Gamma_1(v)$. These vertices are $r$-regular in $G$, the addition of their degrees in $G$ is $r^2$. In $H$ we have to substract the $r$ connections to vertex $v$ and the $r-1$ connections to the vertices in $\Gamma_2(v)$. Thus, the addition of the degrees of vertices in $H$ is $r^2-r-(r-1)=(r-1)^2$. From the handshake lemma we know this has to be an even number, which implies that $r$ needs to be odd.

    The distance tree hanging from $v$ must be similar to the sketch in Figure \ref{Proof2}, where the connections inside the dashed line box are unknown explicitly but we know that
    \begin{enumerate}
        \item every vertex $u\in\Gamma_1(v)={1,\ldots,d}$ has at least one connection with another member of $\Gamma_1(v)$.
        \item every vertex $w\in\Gamma_2(v)={1',\ldots,(d-1)'}$ hangs from one and only one vertex from $\Gamma_1(v)$ and has $d-1$ vertices hanging from itself that belong to $\Gamma_3(v)$.
        \item every vertex $x\in\Gamma_3(v)$ has one and only one edge incident to a vertex in $\Gamma_2(v)$ and the rest of them ($d-1$ in total) adjacent to other vertices in $\Gamma_3(v)$.
    \end{enumerate}
    \begin{figure}[htbp]
    \centering\includegraphics[width=0.6\textwidth]{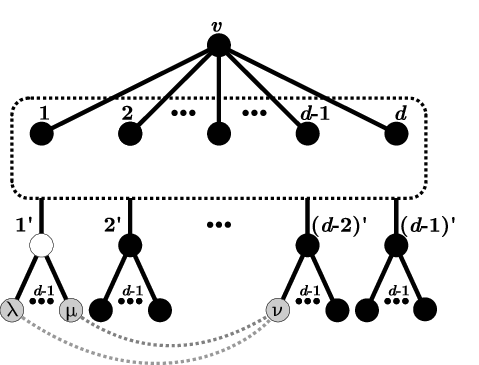}
    \caption{Sketch of the distance tree hanging from $v$.}
    \label{Proof2}
    \end{figure}

    Since $G$ is a radial Moore graph it must have a central vertex of eccentricity 2. From fact 1 we know that it must belong to $\Gamma_2(v)$. It cannot be $v$ or any vertex in $\Gamma_1(v)$ since a central vertex cannot belong to a triangle, and fact 1 shows that all of these vertices belong least one triangle. It cannot be in $\Gamma_3(v)$ since this would imply that it would see $v$ at distance 3. Without loss of generality let's suppose $1'$ is the central vertex. From fact 3 we know that the descendants of $1'$ will be incident only to $1'$ and other vertices in $\Gamma_3(v)$. What is more, since $1'$ is central, no edges are allowed between them. This implies a total of $(d-1)^2$ edges to be connected to a total of $(d-1)(d-2)$ vertices and thus there must exist a vertex $\nu\in\Gamma_3(v)$ not a descendant from $1'$ and vertices $\lambda,\mu\in\Gamma_3(v)$ descendants from $1'$ such that $\lambda\nu$ and $\mu\nu$ are both edges of $G$. This is a contradiction since on one side $1'$ is central and there are two different paths of length 2 from $1'$ to $\nu$. The first part of the proof along with this contradiction prove that $\Gamma_2(v)\geq d$. \\ % fin primera parte de la demo

    If $v$ is a neighbor of the only central vertex $c$ of $G$, the structure in Figure \ref{Proof3} is forced in order to attain the minimum possible number of vertices in $\Gamma_2(v)$. The $d-1$ vertices of $\Gamma_1(c)\setminus\{v\}$ will inevitably be in $\Gamma_2(v)$. Every two different vertices in $\Gamma_1(v)\setminus\{c\}$ are forced to share an edge to reduce the descendants at distance 2 to the minimum, that is, the induced subgraph generated by $\Gamma_1(v)\setminus\{c\}$ must be a complete graph. There must be $\alpha$ descendants of $\Gamma_1(v)\setminus\{c\}$ with $1\leq\alpha\leq d-1$ to take the $d-1$ remaining edges (one for every vertex in $\Gamma_1(v)\setminus\{c\}$). Thus it becomes clear that $\Gamma_2(v)\geq(d-1)+\alpha$. 
    
    \noindent Now, let us find some restrictions on $\alpha$. Let $g_i$ be the number of vertices of $\Gamma_1(v)\setminus\{c\}$ incident with vertex $i$ for $i=1,\ldots,\alpha$. Then, clearly $1\leq g_i\leq d-\alpha$ and that 
    \begin{equation}\label{eqsum}
        \sum_{i=1}^\alpha{g_i}=d-1.
    \end{equation}
    For every vertex in $\{1,\ldots,\alpha\}$ there must be an adjacency with a vertex in $\Gamma_1(c)\setminus\{v\}$, since $c$ is central and must see vertices $\{1,\ldots,\alpha\}$ in no more than two steps. Every vertex $i\in\{1,\ldots,\alpha\}$ has $d-g_i-1$ extra edges that can be adjacent either to other vertices in $\{1,\ldots,\alpha\}$ or to the $(d-1)^2-\alpha$ descendants of $\Gamma_1(c)\setminus\{v\}$. Gathering all the information we have, the following can be stated about the $(d-1)^2-\alpha$ descendants of $\Gamma_1(c)\setminus\{v\}$ regarding every vertex $i\in\{1,\ldots,\alpha\}$:

    \begin{figure}[htbp]
    \centering\includegraphics[width=0.7\textwidth]{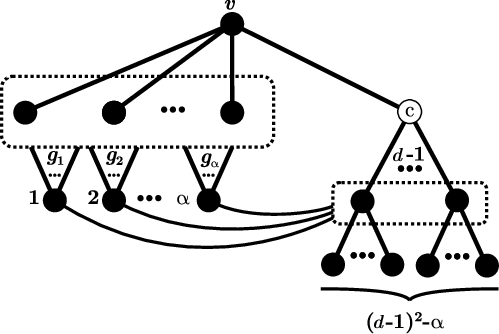}
    \caption{Sketch of the distance tree that must have a neighbor of a central vertex with maximum status in a radial Moore graph.}
    \label{Proof3}
    \end{figure}
    \begin{enumerate}
        \item A maximum of $d-g_i-1$ can see vertex $i$ in one step through a direct edge.
        \item A maximum of $d-2$ vertices can see vertex $i$ in two steps via its ascendant.
        \item A maximum of $(d-g_i-1)(d-2)$ can see vertex $i$ in 2 steps through an adjacency with one of the $d-g_i-1$ direct neighbors in the same level.
        \item The number of vertices that see $i$ in 3 steps are forced to see some $i'\in\{1,\ldots,\alpha\}$, $i\neq i'$ in one step in order to reach the ascendants of $i$ in 3 steps (remember that for every vertex in $\Gamma_1(v)\setminus\{c\}$ only one edge goes to some vertex $\{1,\ldots,\alpha\}$), so there is a total of $$\sum_{j=1,j\neq i'}^\alpha(d-g_j-1).$$
    \end{enumerate}
    These imply the following inequalities for every $i\in\{1,\ldots,\alpha\}$:
    \begin{equation}
       (d-2)+(d-g_i-1)(d-2)+\sum_{j=1}^\alpha(d-g_j-1)\geq(d-1)^2-\alpha
    \end{equation}
    where the terms coming from statements 1 and 4 have been gathered in the same summation. Developing this inequality, and taking into account (\ref{eqsum}) we get
    \begin{equation*}
       g_i\leq\frac{d(\alpha-1)}{d-2}.
    \end{equation*}
    Since it was also stated that $g_i\leq d-\alpha$, the following holds in general
    \begin{equation*}
       g_i\leq\min\left(d-\alpha,\frac{d(\alpha-1)}{d-2}\right).
    \end{equation*}
    This can be also written 
    \begin{equation*}
       g_i\leq
       \left\{
            \begin{array}{lr}
                \frac{d(\alpha-1)}{d-2} & \text{if } \alpha<\frac{d}{2}\\
                d-\alpha & \text{if } \alpha\geq\frac{d}{2}
            \end{array}
        \right.
    \end{equation*}
    This means that depending on the value of $d$ and $\alpha$, there $g_i$ reaches a maximum value that cannot be exceeded in order that the diameter of the graph is 3. Let 
    \begin{equation}\label{gmax}
       g_{max}=
       \left\{
            \begin{array}{lr}
                \left\lfloor\frac{d(\alpha-1)}{d-2}\right\rfloor & \text{if } \alpha<\frac{d}{2}\\
                d-\alpha & \text{if } \alpha\geq\frac{d}{2}
            \end{array}
        \right.
    \end{equation}
    be the maximum possible, according to (\ref{requisito}), number of edges incident from vertices $\Gamma_1(v)\setminus\{c\}$ to a vertex  $i=1,\ldots,\alpha$. In order that all $\Gamma_1(v)\setminus\{c\}$ vertices have one adjacency to vertices $i=1,\ldots,\alpha$ it must hold 
    \begin{equation}\label{requisito}
        \alpha\cdot g_{max}\geq d-1.
    \end{equation}
    If $\alpha<d/2$ this condition is written
    \begin{equation}\label{ineq}
        \alpha\cdot\left\lfloor\frac{d(\alpha-1)}{d-2}\right\rfloor\geq d-1.
    \end{equation}
    If $\alpha\geq d/2$ it can be easily checked that the inequality (\ref{requisito}) always holds, so this does not provide an extra constraint. Since in this case (\ref{ineq}) also holds, it can be stated as a general condition for $\alpha$. Now, since $\lceil \frac{d}{d-2} \rceil =1$ for $d>4$, the minimum value for $\alpha$ satisfying \eqref{ineq} is the ceiling of the largest root of the polynomial $x^2-x-(d-1)$, which is $\left\lceil \frac{1+\sqrt{4d-3}}{2}\right\rceil$ (and in fact this value also coincides for $d=3$ and differs just by one unit for $d=4$).
    \end{proof}

%\begin{figure}[h]
%\centering\includegraphics[width=\textwidth]{ceilingfu.png}
%\caption{Values of $\left\lceil \frac{1+\sqrt{4d-3}}{2}\right\rceil$ for the first 40 values of $d$.}
%\label{alphaR}
%\end{figure}

    The previous lemma can be used to obtain an upper bound on the status of any vertex of a radial Moore graph of diameter $3$.

\begin{corollary}
     The status of a vertex $v$ in a radial Moore graph of degree $d$ and diameter $3$ satisfies $s(v)\leq 3d(d-1)$. Moreover, if the graph contains one central vertex and $v$ is a neighbor of it then
     \begin{equation}\label{UpBound}
        s(v)\leq3d^2-3d-\left\lceil \frac{1+\sqrt{4d-3}}{2}\right\rceil+1.
    \end{equation}
for $d\neq 4$, and $s(v) \leq 35$ for $d=4$.
\end{corollary}

\begin{proof}
    %Since $G$ is $r$-regular, we have $\Gamma_1(v)=r$. If $v$ is a neighbor of a central vertex $\alpha$, the sketch in figure \ref{Proof1} is forced and thus the other $r-1$ descendants of $\alpha$ different from $v$ must be at distance 2 from $v$. 
    The maximum status of a vertex in a radial Moore graph of diameter $3$ will be attained when the number of neighbors at distance two is minimum. According to the previous lemma, we know that $|\Gamma_2(v)|\geq d$. A vertex with maximal status would be one such that $|\Gamma_1(v)|=d$, $|\Gamma_2(v)|=d$ and $|\Gamma_3(v)|=d(d-2)$ and thus, the status would be $d+2d+3d(d-2)=3d(d-1)$. The same counting argument with the corresponding improved bound for neighbors of the central vertex produces the other result.  
\end{proof}

The graph depicted in Figure \ref{SharpBound} has a vertex attaining this upper bound. 
\begin{figure}[htbp]
    \centering\includegraphics[width=0.45\textwidth]{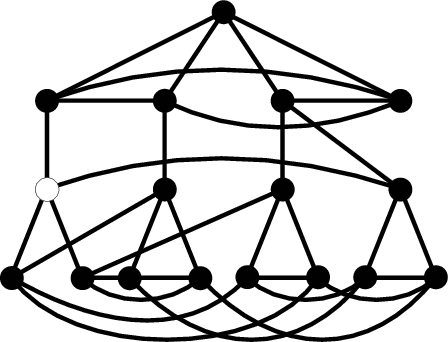}
    \caption{Radial Moore graph hanging from a vertex that holds the equality in \eqref{UpBound} with status 36, proving that, at least for degree $4$ there are examples of radial Moore graphs with vertices attaining this bound. The only central vertex is drawn in white and all their neighbors have status $\leq 35$.}
    \label{SharpBound}
\end{figure}
We recall that the status $s(G)$ of a graph $G$ is the sum of the status of all of its vertices. For instance, the status of a Moore graph of degree $d$ and diameter $3$ is $2d^2-d$. This is precisely a lower bound for the status of a radial Moore graph of the same parameters. An upper bound for the value of $s(G)$ is provided next.

\begin{corollary}
    The status of a radial Moore graph $G$ of degree $d$ and diameter $3$, containing one central vertex, satisfies
    \begin{equation}\label{eq:totalstatusbound}
        s(G)\leq 3d^4-3d^3+2d^2-d\cdot\left\lceil \frac{1+\sqrt{4d-3}}{2}\right\rceil
    \end{equation}
\end{corollary}

\begin{proof}
   We have that $s(G)=\sum_{v_i\in V(G)}s(v_i)$. Also, $G$ has order $d^2+1$ and it must contain at least one central vertex (with status $d(2d-1)$). The upper bound \eqref{UpBound} for the neighbors of this central vertex and $s(v)\leq 3d(d-1)$ for the remaining ones gives the result.
\end{proof}

\begin{figure}[htbp]
\centering\includegraphics[width=0.3\textwidth]{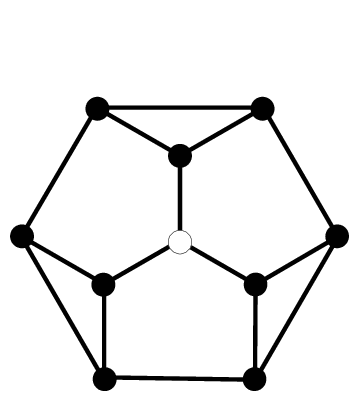}\includegraphics[width=0.38\textwidth]{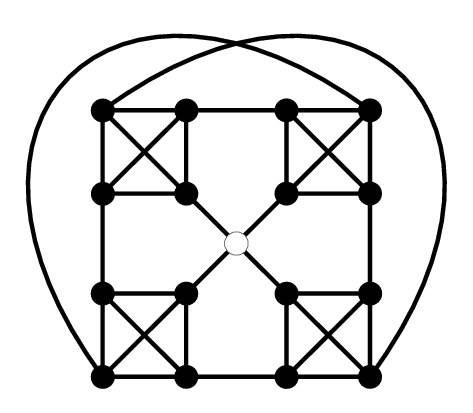}
\caption{Radial Moore graphs with maximum status for $d=3$ (left) and $d=4$ (right) respectively. The central vertex is drawn white. They correspond respectively to $G_3$ and $G_4$.}
\label{MaxStat}
\end{figure}
%\section{An extremal family of graphs}\label{sec:family}
%%%%%%%%%%%%%%%%%%%%%%%%%%%%%%%%%%%%%%%%%%%%%%%%%%%%%%%%%%%%%%%%%%%%%%%

Note that this general upper bound is obtained when vertices at distance two from the central vertex have all maximum local status $3d(d-1)$. Nevertheless, if we take a look to radial Moore graphs of diameter $3$ and degrees $d=3$ and $d=4$ with greater status (shown in Figure \ref{MaxStat}) we see the following: these two graphs have status vector ${\mathbf s_1}:15^1,17^9$ and $\mathbf{s_2}:28^1,34^{16}$, respectively. Both have a single central vertex and all the other vertices share the same status. Their total status is, respectively, $s(G_1)=168$ and $s(G_2)=572$. Besides, the upper bound \eqref{eq:totalstatusbound} produces $s(G)\leq 174$ and $s(H) \leq 596$ for degrees $3$ and $4$, respectively. These small gaps are consequence from the fact that they do not have vertices with local status $3d^2-3d$. These two graphs can be expressed as members of a general family of graphs where every non central vertex has local status $3d^2-4d+2$ instead. One way to construct them is the following:

\begin{enumerate}
    \item Let $d \geq 3$, and let us take $d$ disjoint copies of the complete graph $K_{d-1}$ that we will denote as $H_1,H_2,\dots,H_d$.
    \item Attach a $1$-factor to this graph by adding a matching between vertices of different copies of these complete graphs in such a way that  there exists two unique adjacent vertices $u \in V(H_i)$ and $v \in V(H_j)$ joining $H_i$ and $H_j$ for all $i \neq j$. %This produces an $r-1$-regular graph.
    \item Finally take the bipartite graph $K_{1,d}$ with partite sets $V_0=\{0\}$ and $V_1=\{1,2,\dots,d\}$ and joint every vertex $i$ to each vertex of $H_i$, for all $i=1,2,\dots,d$.
\end{enumerate}
Roughly speaking, if we hang this graph from vertex $0$, we see a Moore tree of high $2$, where every neighbor $i$ of $0$ is attached to the complete graph $H_i$, for all $i=1,\dots,d$, and all of these complete graphs have a `unique' connection between them. Let us see the status vector of this construction.
\begin{proposition}
Any graph $G_d$ constructed as above is a radial Moore graph of degree $d$, diameter $3$ and it has vector status 
$$
{\bf s}(G_d): (d^2-d)^1, (3d^2-4d+2)^{d^2}.
$$
\end{proposition}
\begin{proof}
$G_d$ is a clearly an $d$-regular graph of order $1+d^2$ by construction (it begins with an ($d-2$)-regular graph of order $d(d-1)$ at step $1$ and we add one edge to every vertex $u \in V(H_i)$ at each step). The new $1+d$ vertices added at step $3$ have also degree $d$. Vertex $0$ has eccentricity $2$ since every vertex in $H_i$ is reached from $0$ through vertex $i$, for all $i=1,\dots,d$ and hence it has status $d^2-d$. Any vertex $i$, where $i\in \{1,2,\dots,d\}$ has eccentricity $3$. To see this, observe that vertex $0$ and all the vertices of $H_i$ are adjacent to $i$ and hence at distance $2$ we have $2(d-1)$ vertices (all the vertices $\{1,2,\dots,d\}$ (except $i$)  plus $d-1$ vertices corresponding to the unique vertex $v \in V(H_j)$, $j \neq i$, that it has every vertex $u \in V(H_i)$ by definition at step $2$). The remaining vertices (the neighbors of $u$ inside each complete graph) are at distance $3$. This produces a status of $3d^2-4d+2$. A similar reasoning produces the same status for all the vertices inside each $H_i$ graph. Curiously enough, vertex $i$ and every vertex in $H_i$ share the same status, although the corresponding rooted graphs are different (but they have the same number of vertices at each distance). 
\end{proof}
In particular $G_d$ has total status $s(G_d)=3d^4-4d^3+3d^2-d$ which is asymptotically comparable with the upper bound \eqref{eq:totalstatusbound}.
An algebraic construction for $G_d$ is given by labelling the vertices of each $H_i$ as 
$$V(H_i)=\{(i,j) \, | \,  j=1,\dots,d; \, j\neq i\},$$  (see Fig. \ref{Esq}) and then the 'unique' connection between two different complete graphs $H_i$ and $H_k$ is given by $(i,k)\sim(k,i)$ for all $i\neq k$.
\begin{figure}[htbp]
\centering\includegraphics[width=0.85\textwidth]{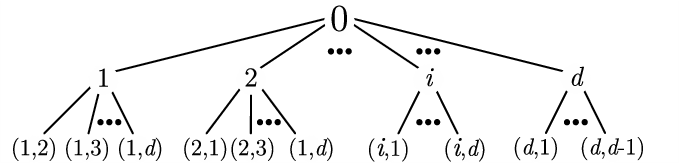}
\caption{Labelling of the Moore tree of $G_d$ hanging from its unique central vertex $0$.}
\label{Esq}
\end{figure}

There are many symmetries in the graph $G_d$, as one may expect looking at the examples $G_3$ and $G_4$ showed in Fig. \ref{MaxStat}. Notice that every $G_d$ contains $d$ different copies of the complete graph $K_d$ (the subgraph induced by vertex $i$ and the vertices belonging to $H_i$, for all $i=1,\dots,d$). It is well known that the automorphism group of $K_d$ is precisely the permutation group of $d$ elements (also known as the symmetric group $S_d$), and this is precisely what retains the full graph $G_d$ from the symmetries of its induced complete subgraphs, as next proposition shows. 

\begin{proposition}
The automorphism group of $G_d$ is the symmetric group of order $d$, that is, $\textrm{Aut}(G_d) \cong S_d$.
\end{proposition}

\begin{proof}
Let us consider the labeling of the vertices of $G_d$ as given above (see Fig. \ref{Esq}). Let $\varphi_{i,j} \in S_d$ be the transposition $(i \ j)$, that is, the permutation of the set $\{1,2,\dots,d\}$ which exchanges elements $i$ and $j$ and keeps all others fixed, for all $i\neq j$. For each transposition $(i \ j)$ we define the mapping $\Phi_{ij}:V(G_d)\rightarrow V(G_d)$ such that:
\[
    \begin{array}{l}
        \Phi_{ij}(0)=0;\\
        \Phi_{ij}(a)=\varphi_{ij}(a), \ \forall a\in\{1,\dots,d\};\\
        \Phi_{ij}((a,b))=(\varphi_{ij}(a),\varphi_{ij}(b)), \ \forall a,b\in\{1,\dots,d\}, \ a\neq b.
    \end{array}
\]
Clearly $\Phi_{ij}$ is a graph automorphism of $G_d$ since all the adjacencies are preserved. As a consequence, every permutation of $S_d$ induces an automorphism of $G_d$. \\

On the other hand, given an automorphism $\phi$ of $G_d$ we have that $\phi(0)=0$ since it is the unique vertex of eccentricity $2$. Besides, every vertex $v \in \Gamma_l(0)$ must be mapped into the same set $\Gamma_l(0)$, for $l=1,2$, since the distance preserving spanning tree of any vertex in $\Gamma_1(0)$ is different from the one of a vertex in $\Gamma_2(0)$. Hence, the automorphism $\phi$ restricted for the set of vertices in $\Gamma_1(0)$ defines a permutation $\sigma \in S_d$  where $\sigma(i)=\phi(i)$ for all $i \in \{1,\dots,d\}(=\Gamma_1(0))$. Now we will see that $\sigma$ determines also the mapping $\phi$ for the vertices in $\Gamma_2(0)$. Indeed, since $i \sim (i,j)$, it turns out that for every $j \neq i$, it must exist $k \neq \sigma(i)$ such that $\phi((i,j))=(\sigma(i),k)$. In addition $(i,j) \sim (j,i)$ and hence $\phi((i,j)) \sim \phi((j,i))$, that is, $(\sigma(i),k) \sim (\sigma(j),k')$, that is, $k=\sigma(j)$. Hence, $\phi((i,j))=(\sigma(i),\sigma(j))$ proving that any automorphism of $G_d$ induces a permutation of $S_d$.
\end{proof}

\section{Final remarks and open problems}\label{sec:conc}

Radial Moore graphs appear as an approximation to Moore graphs when the diameter condition is relaxed. The status of a vertex can be useful to test how good is this approach. In this paper we present an upper bound on the status of a radial Moore graph. We start by applying two already-known results recursively to obtain a lower bound on the number of non-central vertices. Then, we focus on the case of diameter $3$  where we determine an upper bound in the status of every single vertex that is tight for some parameters of the degree $d$. An extra upper bound is presented for specific vertices. Finally, an infinite family of radial Moore graphs of diameter $3$ with high status is presented and we strongly believe that is maximal in terms of the status. 

\begin{conjecture}
    The radial Moore graph $G_d$ has maximal status for all $d \geq 3$.
\end{conjecture}

The upper bound for the maximum number of central vertices in a radial Moore graph presented in Prop. \ref{prop:centralvertices} is poor for $k=2$, since it provides the value ${\cal C}(d,2) \leq M(d,2)-6$ which seems to be far from the exact number of ${\cal C}(d,2)$ for large $d$. Even for $d=7$, where a Moore graphs exists, ${\cal C}(7,2)$ seems to be $28$ (far from $50-6=44$) which corresponds to make a particular edge swap to the Hoffman-Singleton graph.

\begin{problem}
    Improve the bounds for ${\cal C}(d,k)$, specially for the case $k=2$.
\end{problem}

\bibliographystyle{unsrt}
\bibliography{biblio}

\begin{thebibliography}{10}

\bibitem{Hoffman1960}
A.~J. Hoffman and R.~R. Singleton.
\newblock On moore graphs with diameters 2 and 3.
\newblock {\em IBM Journal of Research and Development}, 4(5):497--504, 1960.

\bibitem{MS12}
M.~Miller and J.~Ŝir\'an.
\newblock {M}oore graphs and beyond: {A} survey of the degree/diameter problem.
\newblock {\em Electronic Journal of Combinatorics}, 11:1 -- 89, 2012.

\bibitem{CCEGL2010}
C.~Capdevila, J.~Conde, G.~Exoo, J.~Gimbert, and N.~López.
\newblock Ranking measures for radially {M}oore graphs.
\newblock {\em Networks}, 56(4):255--262, 2010.

\bibitem{EXOO20121507}
G.~Exoo, J.~Gimbert, N.~López, and J.~Gómez.
\newblock Radial moore graphs of radius three.
\newblock {\em Discrete Applied Mathematics}, 160(10):1507--1512, 2012.

\bibitem{Knor2007SmallRadially}
M.~Knor.
\newblock Small radially moore graphs.
\newblock In {\em Mathematics, Geometry and their Applications}, pages 59--62,
  Ko\v{s}covce, 2007.
\newblock Presented in 2007.

\bibitem{knor2012}
Martin Knor.
\newblock Small radial moore graphs of radius 3.
\newblock {\em Australas. J Comb.}, 54:207--216, 2012.

\bibitem{GOMEZ201515}
J.~Gómez and M.~Miller.
\newblock On the existence of radial moore graphs for every radius and every
  degree.
\newblock {\em European Journal of Combinatorics}, 47:15--22, 2015.

\bibitem{knor1996}
Martin Knor.
\newblock A note on radially moore digraphs.
\newblock {\em IEEE transactions on computers}, 45(3):381--382, 1996.

\bibitem{CLC2023}
J.~M. Ceresuela, N.~L\'opez, and D.~Chemisana.
\newblock On mixed radial moore graphs of diameter 3.
\newblock {\em Discrete Mathematics}, 346:113525, 2023.

\bibitem{BuckHara}
F.~Buckley and F.~Harary.
\newblock {\em Distance in graphs}.
\newblock The Advanced Book Program. Addison-Wesley Pub. Co., 1990.

\bibitem{LG10}
Nacho López and Joan Gimbert.
\newblock Vértices centrales en los grafos radiales de moore cúbicos.
\newblock In {\em Proceedings of the VII Jornadas de Matemática Discreta y
  Algorítmica}, pages 425--434, 2010.

\bibitem{GL08}
Joan Gimbert and Nacho López.
\newblock Sobre la existencia de grafos radiales de moore.
\newblock In {\em Jornadas de Matemática Discreta y Algorítmica}, pages
  361--368, 2008.

\bibitem{zbMATH01820648}
Q.~I. Rahman and G.~Schmeisser.
\newblock {\em Analytic theory of polynomials}, volume~26 of {\em Lond. Math.
  Soc. Monogr., New Ser.}
\newblock Oxford: Oxford University Press, 2002.

\end{thebibliography}
\end{document}